\documentclass[a4paper,11pt]{amsart}%
\usepackage[utf8]{inputenc}
\usepackage[utf8]{inputenc}
\usepackage{enumitem}
\usepackage{amssymb,amsthm, amsfonts}
\usepackage{amsmath}
\usepackage{amsthm}
\usepackage{algorithm}
\usepackage{algorithmic}
\usepackage{graphicx}

\usepackage[backend=biber,style=numeric]{biblatex}
\usepackage[margin = 2.54cm]{geometry}
\graphicspath{ {images/} }
\theoremstyle{plain}
\newtheorem{theorem}{\bf Theorem}[section]
\newtheorem{lemma}[theorem]{\bf Lemma}

\newtheorem{assumption}[theorem]{\bf Assumption}

\theoremstyle{definition}

\theoremstyle{remark}
\newtheorem*{remark*}{\bf Remark}
\newtheorem{remark}[theorem]{\bf Remark}

\makeatletter\let\saved@bibitem\@bibitem\makeatother
\usepackage[pdftex,hidelinks]{hyperref}
\makeatletter\let\@bibitem\saved@bibitem\makeatother

\title[Deep multi-step mixed algorithm for high-dimensional BSDEs]{Deep multi-step mixed algorithm for high dimensional non-linear PDEs and associated BSDEs}
\author{Daniel Bussell}
\address{}
\email{daniel.bussell.14@ucl.ac.uk}
\author{Camilo Andr\'es Garc\'ia Trillos}
\address{}
\email{camilo.garcia@ucl.ac.uk}

\begin{document}

\maketitle

\begin{abstract}
We propose a new multistep deep learning-based algorithm for the resolution of moderate to high dimensional nonlinear backward stochastic differential equations (BSDEs) and their corresponding parabolic partial differential equations (PDE). Our algorithm relies on the iterated time discretisation of the BSDE and approximates its solution and gradient using deep neural networks and automatic differentiation at each time step. The approximations are obtained by sequential minimisation of local quadratic loss functions at each time step through stochastic gradient descent. We provide an analysis of approximation error in the case of a network architecture with weight constraints requiring only low regularity conditions on the generator of the BSDE. The algorithm increases accuracy from its single step parent model and has reduced complexity when compared to similar models in the literature. 
    
\end{abstract}

\begin{section}{Introduction}
Let us consider a non-linear decoupled forward backward stochastic differential equation (BSDE) of the form
\begin{equation}\label{e:BSDE}
    \begin{cases}
    dX_{t}&=b(t,X_{t})dt+\sigma(t,X_{t})dW_{t},\,\ X_{0}=x_{0},\\
    -dY_{t}&=f(t,X_{t},Y_{t},Z_{t})dt-Z_{t}dW_{t},\\
    Y_{T}&=g(X_{T}),\\
    
    \end{cases}
\end{equation}
for $b\in\mathbb{R}^{d},\sigma\in\mathbb{M}^{d}$ where $\mathbb{M}^{d}$ is the set of $d\times d$ real valued matrices, a nonlinear generator function $f:[0,T]\times\mathbb{R}^{d}\times\mathbb{R}\times\mathbb{R}^{d}\to\mathbb{R}$, $W_{t}=(W_{t}^{1},\dots W_{t}^{d})$ is a $d$-dimensional Brownian motion and $g:\mathbb{R}^{d}\to\mathbb{R}$ is a terminal function. A solution to \eqref{e:BSDE} is a triple of adapted processes $(X,Y,Z)$ typically required to satisfy certain integrability conditions, such that \eqref{e:BSDE} holds. BSDEs are connected to semi-linear partial differential equations through the nonlinear Feynman-Kac formula, and indeed can be understood as stochastic analogues of such semi-linear partial differential equations. They have a wide range of applications in optimal control, economics and mathematical finance. To mention only a few examples, they can be used in European options pricing \cite{karoui1997}, American options pricing \cite{bouchard2012} and utility maximisation \cite{kobylanski2000}. 

Several existence and uniqueness results for solutions of \eqref{e:BSDE} have been obtained under different structural assumptions and integrability conditions (for example, the Lipschitz case on $\mathcal S^2 \times \mathcal H^2$ is studied in \cite{pardoux1990}). However, BSDEs do not usually have explicit analytically tractable solutions and this has motivated the wide study of numerical methods of approximation. Such methods have been extensively studied in \cite{gobet2009},\cite{bender2008},\cite{chassagneux2014},\cite{ma2008} and \cite{camilo2015} amongst many others. While performing well on low dimensions,  many of these methods suffer from a ``curse of dimensionality" in which their computational complexity increases exponentially with the dimension of the state variable. 

Recently, several works have aimed to avoid the curse of dimensionality by using neural networks methods to solve  \eqref{e:BSDE} in high-dimensional settings. They essentially rely on finding approximations to the \emph{decoupling field}, that is, a representation in terms of functions of the state variables. More specifically, the process of finding a solution to the BSDE is recasted as a reinforcement learning task, by  re-expressing a weaker form of equation \eqref{e:BSDE} as a loss function ruling the training of the introduced neural networks. Examples of works in this area include the Deep BSDE method in \cite{jentzen2017} or the Deep Dynamic Programming (DDP) methods presented in \cite{hure2019}.  We include a short reminder of these methods in Section \ref{s:ExistingLearningBSDE} below. 

In this paper we study a multi-step version of the DDP algorithm in \cite{hure2019}. We use a neural network to approximate the decoupling field for the `Y' variable and automatic differentiation to approximate the `Z' variable. Networks are then trained to minimise a quadratic loss function involving evaluations of the approximation at more than one time step. This loss function exploits an idea first studied in \cite{gobetturk2016} to modify the now classical one step scheme to decrease the overall complexity in the approximation. The scheme is presented in \eqref{MultistepLoss}.

Our main result, Theorem \ref{t:ErrorAnalysis},
 is a proof of the consistency and convergence of the scheme. It shows that passing from a one-step to a multi-step scheme reduces the demands on accuracy for the trained neural networks. Moreover, we are able to avoid high regularity conditions on the generators $f$ as the ones imposed in \cite{hure2019}.  
 
 Our work is related to \cite{ger2020} where an analogous approach was studied using  two networks, instead of one, per each time iteration to approximate independently the variables $Y,Z$. Mathematically, both our method and \cite{ger2020} achieve similar improvements in accuracy over their single-step pairs, but crucially we do not train independently another variable. 
 Although parts o the development are similar, we need the fact that $Z$ is not trained independently requires additional requirements in the analysis.

The rest of this paper is organized as follows:
Section 2 serves as an introduction to basic theory on the connection between BSDEs and PDEs and an introduction to neural networks. Section 3 introduces some existing schemes on solving PDEs using neural networks from the literature. Section 4 introduces the Deep Automatic Differentiation Multistep (DADM) Scheme. Section 6 provides an approximation error analysis of the DADM scheme. Section 6 illustrates numerical tests of the DADM scheme accuracy in comparison with existing methods.

\end{section}
\begin{section}{Preliminaries}
Let $(\Omega,\mathcal{F},\mathbb{P},\{\mathcal{F}_{t}\}_{0\leq t\leq T})$ be a filtered probability space. On this space we define a $d$-dimensional Brownian motion $W_{t}$ such that the filtration $\{\mathcal{F}_{t}\}_{0\leq t\leq T}$ is the natural filtration of $W_{t}$. We define $L^{2}=L^{2}_{\mathcal{F}}(0,T;\mathbb{R}^{d})$ the set of all $\mathcal{F}_{t}$-adapted and square integrable processes valued in $\mathbb{R}^{d}$. The triple $(X_{t},Y_{t},Z_{t}):[0,T]\times\Omega\to\mathbb{R}^{d}\times\mathbb{R}\times\mathbb{R}^{d}$ is said to be a solution of \eqref{e:BSDE} if it is $\mathcal{F}_{t}$ adapted, square integrable and satisfies \eqref{e:BSDE}.

When a solution of \eqref{e:BSDE} exists, under the current Markovian setting, we can use the flow of the diffusion $X$ to find a measurable function $u$ such that
\begin{equation}
u(t,X_t) := Y_t
\label{e:func_rep_Y}
\end{equation}

See for example \cite{pardoux1992}.
This representation of the solution of the BSDE as a function motivates our choice of neural networks as a function approximator.

\begin{subsection}{Nonlinear Feynman-Kac formula}
BSDEs can be used to provide a probabilistic interpretation for the solution of the semilinear parabolic PDE:
\begin{equation}\label{e:PDE}
 \frac{\partial u}{\partial t}+\sum_{i=1}^{d}b_{i}(t,x)\frac{\partial u}{\partial x_{i}}+\frac{1}{2}\sum_{i,j=1}^{d}(\sigma\sigma^{T})_{i,j}(t,x)\frac{\partial^{2}u}{\partial x_{i}\partial x_{j}}+f(t,x,u,(\nabla u)\sigma)=0   
\end{equation}
with terminal condition $u(T,x)=g(x)$. As shown in \cite{pardoux1992}, under global Lipschitz conditions on $b,\sigma,f,g$ uniformly in $t$ (for $f$), $u$ defined by \eqref{e:func_rep_Y} is the unique viscosity solution to the above PDE. Conversely, assuming that \eqref{e:PDE} has a classical solution $u\in \mathcal{C}^{1,2}([0,T]\times\mathbb{R}^{d})$, the solution of \eqref{e:BSDE} is given by 
\begin{equation}\label{e:PBSoln}
 Y_{t}=u(t,X_{t}),\,\ Z_{t}=(\nabla u(t,X_{t}))\sigma(t,X_{t})\,\ t\in[0,T],   
\end{equation}
$\mathbb{P}$ almost surely. 
\end{subsection}
\begin{subsection}{Approximation of functions by neural networks}
Deep neural networks are a class of functions designed to approximate unknown functions. They are constructed by the composition of simple maps and affine transformations, and they provide a way to deal with high-dimensional approximation problems in an efficient manner.
\\
We shall work with feed-forward neural networks. We fix the input dimension $d$, the output dimension $d_{1}$, the number $L+1\in \mathbb{N}$ of layers with $m_{l},l=0,\dots,L$ the number of neurons per layer. The first layer is the input layer with $m_{0}=d$,the final layer is the output layer with $m_{L}=d_{1}$ and the $L-1$ inner layers are called hidden layers with $m_{i}=m,i=1,\dots,L-1$.
\\
We then define a feed-forward neural network from $\mathbb{R}^{d}$ to $\mathbb{R}$ as the composition 
\begin{equation}\label{NNetDef}
   x\in\mathbb{R}^{d}\to \mathcal{A}_{L}\circ\rho\circ\mathcal{A}_{L-1}\circ\dots\circ\rho\circ\mathcal{A}_{1}(x)\in\mathbb{R}^{d_{1}}, 
\end{equation}
where $\mathcal{A}_{i},i=1,\dots,L$ are a sequence of linear transformations represented by 
\begin{equation*}
    \mathcal{A}_{i}(x)=\mathcal{W}_{i}x+\beta_{i},
\end{equation*}
for a weight matrix $\mathcal{W}_{i}$, bias vector $\beta_{i}$ and nonlinear activation function $\rho:\mathbb{R}\to\mathbb{R}$ which is applied component-wisely to the outputs of $\mathcal{A}(x)$. By abuse of notation we write $\rho(x)=(\rho(x_{1}),\dots,\rho(x_{d})$.
\\
The matrices $\mathcal{W}_{i}$ and vectors $\beta_{i}$ are the parameters of the neural network and can be identified with an element $\theta\in\mathbb{R}^{N_{m}}$, where $N_{m}=\sum_{i=0}^{L-1}m_{i}(1+m_{i+1})=d(1+m)+m(1+m)(L-2)+m(1+d_{1})$ is the number of parameters. We denote by $\Phi_{m}(.;\theta)$ the neural network function defined in \eqref{NNetDef} and by $\mathcal{NN}^{\rho}_{d,d_{1},L,m}(\mathbb{R}^{N_{m}})$ the set of all such neural networks $\Phi_{m}(.;\theta)$ for $\theta\in\mathbb{R}^{N_{m}}$, and set
\begin{equation*}
    \mathcal{NN}^{\rho}_{d,d_{1},L}=\bigcup_{m}\mathcal{NN}^{\rho}_{d,d_{1},L,m}(\mathbb{R}^{N_{m}})
\end{equation*}
as the class of all neural networks with architecture given by $d,d_{1},L$ and $\rho$.
\\

The following theorem from \cite{hornik1989} justifies our choice of neural networks as function approximators
\begin{theorem}[Universal Approximation Theorem:]
$\mathcal{NN}^{\rho}_{d,d_{1},L}$ is dense in $L^{2}(\nu)$ for any finite measure $\nu$ on $\mathbb{R}^{d}$, whenever $\rho$ is continuous and non-constant
\end{theorem}

For the purpose of convergence analysis, we introduce the following class of networks with one hidden layer,\,\ $C^{4}$ activation function $\rho$ with linear growth condition and bounded derivatives:
\begin{equation}\label{e:WeightDerivNet}
\mathcal{NN}^{\rho}_{d,1,2,m}(\Theta^{\gamma}_{m}):=\left\{x\in\mathbb{R}^{d}\to\mathcal{U}(x;\theta)=\sum_{i=1}^{m}c_{i}\rho(a_{i}x+b_{i})+b_{0},\theta=(a_{i},b_{i},c_{i},b_{0})_{i=1}^{m}\in\Theta_{m}^{\gamma}\right\}    
\end{equation}
where
\begin{equation*}
\Theta_{m}^{\gamma}=\left\{\theta=(a_{i},b_{i},c_{i},b_{0})_{i=1}^{m} | \max_{i=1,\dots,m}|a_{i}|\leq \gamma_{m},\,\ \sum_{i=1}^{m}|c_{i}|\leq\gamma_{m}\right\}
\end{equation*}
for a sequence $(\gamma_{m})_{m}$ converging to $\infty$, as $m\to\infty$, and such that 
\begin{equation}\label{e:GammaSeq}
  \frac{\gamma_{m}^{6}}{N}\to 0,  
\end{equation}
as $m,N\to\infty.$
We present the following lemma on the bounds of the network derivatives:
\begin{lemma}{(Bounds on network derivatives)}\label{NetworkBoundsLemma}
Suppose $\mathcal{U}\in\mathcal{NN}_{d,1,2,m}^{\rho}(\Theta_{m}^{\gamma})$, then there exists a $C>0$, depending only on $d$ and the derivatives of $\rho$, such that
\begin{equation}\label{e:DerivBoundsU}
\begin{cases}
\sup_{x\in\mathbb{R}^{d},\theta\in\Theta_{m}^{\gamma}}&\left|D_{x}\mathcal{U}(x;\theta)\right|\leq C\gamma_{m}^{2},\\
\sup_{x\in\mathbb{R}^{d},\theta\in\Theta_{m}^{\gamma}}&\left|D^{2}_{x}\mathcal{U}(x;\theta)\right|\leq C\gamma_{m}^{3},\\
\sup_{x\in\mathbb{R}^{d},\theta\in\Theta_{m}^{\gamma}}&\left|D^{3}_{x}\mathcal{U}(x;\theta)\right|\leq C\gamma_{m}^{4},\\
\sup_{x\in\mathbb{R}^{d},\theta\in\Theta_{m}^{\gamma}}&\left|D^{4}_{x}\mathcal{U}(x;\theta)\right|\leq C\gamma_{m}^{5}.
\end{cases}   
\end{equation}
\end{lemma}
\begin{proof}
For $\mathcal{U}\in\mathcal{NN}_{d,1,2,m}^{\rho}(\Theta_{m}^{\gamma})$ we differentiate \eqref{e:WeightDerivNet} to obtain
\begin{equation}
D_{x}\mathcal{U}(x;\theta)=\sum_{i=1}^{m}a_{i}c_{i}D_{x}\rho(a_{i}x+b_{i})    
\end{equation}
for $x\in\mathbb{R}^{d},\theta\in\Theta_{m}^{\gamma}$. We then take absolute value and obtain
\begin{equation*}
\begin{aligned}
|D_{x}\mathcal{U}(x;\theta)|&=\left|\sum_{i=1}^{m}a_{i}c_{i}D_{x}\rho(a_{i}x+b_{i})\right|\\
&\leq C\max_{i=1,\dots,m}|a_{i}|\sum_{i=1}^{m}|c_{i}|\leq C\gamma_{m}^{2},
\end{aligned}
\end{equation*}
where we have used the fact that the derivative of $\rho$ is bounded and the definition of the set of weights $\Theta_{m}^{\gamma}$. The remaining bounds on the higher derivatives are proved by repeated application of the above argument.
\end{proof}

\end{subsection}
\begin{subsection}{Time discretization of BSDEs}
We introduce the integral form of the forward diffusion process $X$ in \eqref{e:BSDE} as
\begin{equation}\label{e:ForwardSDE}
X_{t}=X_{0}+\int_{0}^{t}b(s,X_{s})ds+\int_{0}^{t}\sigma(s,X_{s})dW_{s},\,\ 0\leq t\leq T.    
\end{equation}
We consider the time discretization
\begin{equation*}
    \pi=\left\{t_{i}|t_{i}\in[0,T],i=0,1,\dots,n,t_{i}<t_{i+1},\Delta t_{i}=t_{i+1}-t_{i},t_{0}=0,t_{n}=T\right\}
\end{equation*}
of the time interval $[0,T]$ with mesh $|\pi|:=\sup_{i}\Delta t_{i}$ such that $h=O(\frac{1}{n})$. We introduce the Euler-Maruyama scheme $(X^{\pi}_i)_{i=0,\dots,n}$ defined by 
\begin{equation*}
X_{i+1}^{\pi}=X_{0}+b(t_{i},X^{\pi}_{i})\Delta t_{i}+\sigma(t_{i},X^{\pi}_{i})\Delta W_{i}    
\end{equation*}
where $\Delta W_{i}=W_{t_{i+1}}-W_{t_{i}},i=0,\dots,n$.
\\
We write the time discretization of the BSDE \eqref{e:BSDE} in backward form as
\begin{equation}\label{e:BackwardOneStep}
Y_{i}^{\pi}=Y_{i+1}^{\pi}+f(t_{i},X^{\pi}_{i},Y_{i}^{\pi},Z_{i}^{\pi})-Z_{i}^{\pi}\Delta W_{i},\,\ i=0,\dots,n-1
\end{equation}
which we may write as 
\begin{equation}\label{e:BackwardExpectation}
\begin{cases}
Y_{i}^{\pi}&=\mathbb{E}_{i}\left[Y_{i+1}^{\pi}+f(t_{i},X_{i}^{\pi},Y_{i}^{\pi},Z_{i}^{\pi})\Delta t_{i}\right]\\
Z_{i}^{\pi}&=\mathbb{E}_{i}\left[\frac{\Delta W_{i}}{\Delta t_{i}}Y_{i+1}^{\pi}\right]
\end{cases}    
\end{equation}
where $\mathbb{E}_{i}$ is the conditional expectation operator with respect to the filtration $\mathcal{F}_{t_{i}}$.
\end{subsection}
\end{section}
\begin{section}{Some existing deep learning schemes}
\label{s:ExistingLearningBSDE}
In this section we present a review of some existing deep learning schemes for BSDEs.
\begin{subsection}{Deep BSDE Scheme}
This scheme, from \cite{jentzen2017}, is one of the first works applying deep learning methods for the resolution of BSDEs. It is based on a stochastic optimization interpretation of the BSDE: the variable $Y$ is treated as a controlled stochastic process with its initial value $Y_0$ and the process $Z$ being the controls in closed-loop form. More specifically, we look to minimise the \emph{global} square error loss function 
\begin{equation*}
L(\mathcal{U}_{0},\boldsymbol{\mathcal{Z}})=\mathbb{E}|Y^{\mathcal{U}_{0},\boldsymbol{\mathcal{Z}}}_{N}-g(X^{\pi}_{N})|^{2},  
\end{equation*}
over neural network functions $\mathcal{U}_{0}:\mathbb{R}^{d}\to\mathbb{R}$ and sequences of neural network functions $\boldsymbol{\mathcal{Z}}=(\mathcal{Z}_{i})_{i},\,\ \mathcal{Z}_{i}:\mathbb{R}^{d}\to\mathbb{R}^{d},\,\ i=0,\dots,N-1$, where 
\begin{equation*}
Y^{\mathcal{U}_{0},\boldsymbol{\mathcal{Z}}}_{i+1}=Y^{\mathcal{U}_{0},\boldsymbol{\mathcal{Z}}}_{i}+f(t_{i},X^{\pi}_{i},Y^{\mathcal{U}_{0},\boldsymbol{\mathcal{Z}}}_{i},\mathcal{Z}_{i}(X^{\pi}_{i}))\Delta t_{i}+\mathcal{Z}_{i}(X^{\pi}_{i})\Delta W_{i},\,\ i=0,\dots,N-1    
\end{equation*}
where $Y^{\mathcal{U}_{0},\boldsymbol{\mathcal{Z}}}_{0}=\mathcal{U}_{0}(X^{\pi}_{0})$.
The output of the scheme, $\Hat{\mathcal{U}}_{0}$, is an approximation of the solution $Y_{0}$ to the BSDE at time $0$.
\end{subsection}
\begin{subsection}{Deep Backward Dynamic Programming(DBDP)\cite{hure2019}}
This scheme uses the backward discretisation of the BSDE \eqref{e:BackwardOneStep} and learns the pair $(Y_{t_{i}},Z_{t_{i}})$ at each time step using neural networks trained with the forward process $X^{\pi}_{i}$. There are two versions of the scheme that differ in their approximation for the gradient function $Z$:
\begin{enumerate}
    \item DBDP1: The scheme is initialised with $\Hat{\mathcal{U}}_{N}(X^{\pi}_{N})=g(X^{\pi}_{N})$ and proceeds by backward induction for $i=N-1,\dots,0$ by minimising local square error loss functions 
    \begin{equation*}
        \begin{split}
               L_{i}  (\mathcal{U}_{i},\mathcal{Z}_{i})  = \mathbb{E}\Big| & \Hat{\mathcal{U}}_{i+1}(X^{\pi}_{i+1})-\mathcal{U}_{i}(X_{i}^{\pi}) \\
               & -f(t_{i},X_{i}^{\pi},\mathcal{U}_{i}(X_{i}^{\pi}),\mathcal{Z}_{i}(X_{i}^{\pi}))\Delta t_{i}-\mathcal{Z}_{i}(X_{i}^{\pi})\Delta W_{i}\Big|^{2}
        \end{split}
    \end{equation*}
    over the neural network functions $\mathcal{U}_{i}:\mathbb{R}^{d}\to\mathbb{R}$ and $\mathcal{Z}_{i}:\mathbb{R}^{d}\to\mathbb{R}^{d}$. We then set $(\Hat{\mathcal{U}}_{i},\Hat{\mathcal{Z}}_{i})$ as the solution to each local minimisation problem.
    \item DBDP2: The scheme is initialised at $\Hat{\mathcal{U}}_{N}(X_{i}^{\pi})=g(X_{i}^{\pi})$ and proceeds by backward induction for $i=N-1,\dots,0$ by minimising the local square error loss functions 
    \begin{align*}
     L_{i}(\mathcal{U}_{i})=\mathbb{E} \Big|&\Hat{\mathcal{U}}_{i+1}(X^{\pi}_{i+1})-\mathcal{U}_{i}(X_{i}^{\pi})
     -D_{x}\mathcal{U}_{i}(X_{i}^{\pi})^{T}\sigma(t_{i},X_{i}^{\pi})\Delta W_{i}\\
     & \quad  -f(t_{i},X_{i}^{\pi},\mathcal{U}_{i}(X_{i}^{\pi}),\sigma(t_{i},X_{i}^{\pi})^{T}D_{x}\mathcal{U}_{i}(X_{i}^{\pi}))\Delta t_{i} \Big|^{2}   
    \end{align*}
    where $D_{x}\mathcal{U}_{i}$ is the automatic differentiation of the neural network function $\mathcal{U}_{i}$.
    We then set $\hat{\mathcal{U}}_{i}$ as the solution to the local minimisation problem and set $\Hat{\mathcal{Z}}_{i}=\sigma^{T}(t_{i},\cdot)D_{x}\mathcal{U}_{i}$.
    
\end{enumerate}

\end{subsection}
\end{section}
\begin{section}{Deep Automatic Differentiation Multistep Scheme}
To motivate the scheme, we begin by iterating the backward discretisation for the BSDE \eqref{e:BackwardOneStep} and insert the terminal condition $Y^{\pi}_{N}=g(X^{\pi}_{N})$ in order to obtain the following iterated representation 
\begin{equation}\label{e:IteratedBackward}
Y_{i}^{\pi}=g(X_{N}^{\pi})+\sum_{j=i}^{N-1}[f(t_{j},X^{\pi}_{j},Y_{j}^{\pi},Z_{j}^{\pi})\Delta t_{j}-Z_{j}^{\pi}\Delta W_{j}].    
\end{equation}
We will now define a loss function for the training of neural networks based on a weak form of the above scheme: for $i=N-1,\dots,0$, we minimise the local square error loss functions
\begin{equation}\label{MultistepLoss}
\begin{aligned}
L_{i}(\mathcal{U}_{i})=\mathbb{E}\Big|&g(X^{\pi}_{N})+\sum_{j=i+1}^{N-1}f(t_{j},X^{\pi}_{j},\Hat{\mathcal{U}}_{j}(X^{\pi}_{j}),\sigma^{T}(t_{j},X^{\pi}_{j})D_{x}\Hat{\mathcal{U}}_{j+1}(X^{\pi}_{j}))\Delta t_{j}\\
&-\mathcal{U}_{i}(X_{i}^{\pi})-f(t_{i},X_{i}^{\pi},\mathcal{U}_{i}(X_{i}^{\pi}),\sigma(t_{i},X_{i}^{\pi})^{T}D_{x}\mathcal{U}_{i+1}(X_{i}^{\pi}))\Delta t_{i}\\
&-\sum_{j=i+1}^{N-1}\Hat{\mathcal{Z}}_{j}(X^{\pi}_{j})\Delta W_{j}-\sigma(t_{i},X_{i}^{\pi})^{T}D_{x}\mathcal{U}_{i+1}(X_{i}^{\pi})\Delta W_{i}\Big|^{2}   
\end{aligned}
\end{equation}
and we update $\Hat{\mathcal{U}}_{i}$ as the solution to the local minimisation problem and we set $\Hat{\mathcal{Z}}_{i}=\sigma(t_{i},\cdot)^{T}D_{x}\Hat{\mathcal{U}}_{i+1}$. Multistep schemes for the resolution of BSDEs were introduced in \cite{gobetturk2016} and \cite{chassagneux2014} and thus our scheme can be viewed as a deep learning analogue of these schemes. Here, the loss function forces the solution to satisfy the scheme.
\\
\begin{remark}
In \cite{ger2020} the authors introduce a Multistep Deep Backward Algorithm that can be viewed as a Multistep version of DBDP1 from \cite{hure2019}. In this version,  a pair of neural networks $(\mathcal{U}_{i},\mathcal{Z}_{i})$ are used to estimate the solution $(Y_{i},Z_{i})$. Hence, such a scheme requires the computation of two neural network functions $(\Hat{\mathcal{U}}_{i},\Hat{\mathcal{Z}}_{i})$ at each time step, which can be computationally expensive.  In contrast, the method studied in our paper involves automatic differentiation of $\Hat{\mathcal{U}}_{i}$, requiring only the training of the neural networks $\Hat{\mathcal{U}}_{i}$, with potential reductions in computational load.
\end{remark}
In practice, the expectation defining the loss \eqref{MultistepLoss} is replaced by the empirical average over some number of samples $M$. The minimisation of this average is obtained via stochastic gradient descent (SGD) or some of its variants such as Adam. That is, given a timestep $i=N-1,\dots,0$ and sample $(X^{\pi,k}_{j},\Delta W_{j}^{k})_{j=i,\dots N}$ of the forward Euler scheme and Brownian increment for $k=1,\dots,M$, of mini-batch size $M$ the iterations of SGD for timestep $i$ serves to minimise the empirical loss:
\begin{equation}\label{EmpiricalLoss}
\begin{aligned}
 \mathbb{L}_{i}(\theta)=\frac{1}{M}\sum_{k=1}^{M} \Big|&g(X^{\pi,k}_{n})+\sum_{j=i+1}^{n-1}\Hat{\mathcal{Z}}_{j}(X^{\pi,k}_{j})\Delta W_{j}^{k}\\
&+ \sum_{j=i+1}^{n-1}f(t_{j},X^{\pi,k}_{j},\Hat{\mathcal{U}}_{j}(X^{\pi,k}_{j}),\sigma^{T}(t_{j},X^{\pi,k}_{j})D_{x}\Hat{\mathcal{U}}_{j+1}(X^{\pi,k}_{j}))\Delta t_{j}\\
&-\mathcal{U}_{i}(X_{i}^{\pi,k})-f(t_{i},X_{i}^{\pi,k},\mathcal{U}_{i}(X_{i}^{\pi,k}),\sigma(t_{i},X_{i}^{\pi,k})^{T}D_{x}\mathcal{U}_{i+1}(X_{i}^{\pi,k}))\Delta t_{i}\\
&-\sigma(t_{i},X_{i}^{\pi,k})^{T}D_{x}\mathcal{U}_{i+1}(X_{i}^{\pi,k})\Delta W_{i}^{k}\Big|^{2}
\end{aligned}
\end{equation}
where $\hat{\mathcal{U}_{j}}=\mathcal{U}_{j}^{\hat{\theta}_{j}},\hat{\mathcal{Z}}_{j}^{\hat{\theta_{j}}}=\sigma^{T}(t_{j},\cdot)D_{x}\mathcal{U}_{j+1}^{\hat{\theta_{j}}}$, and $\hat{\theta}_{j}$ is the optimal parameter obtained from SGD at future times $j=i+1,\dots,N-1$. In order to reduce computational time, we initialise the parameter of SGD at time $i$ to be the optimal parameter resulting from the SGD at time $i+1$.

\begin{remark}
To ease the notation, we will not include explicitly in our analysis the errors induced by SGD learning and approximating the minimiser with an empirical loss \eqref{EmpiricalLoss}.   These error sources are studied  in \cite{beck2022} and \cite{chizat2018global}. 
\end{remark}

\end{section}

\begin{section}{Error Analysis}
We now begin a theoretical analysis of the approximation error. The methods of proof are inspired by \cite{hure2019} and \cite{ger2020}.

\begin{assumption}\label{LipschitzAssump}
(i) $X_{0}\in L^{3}(\mathcal{F}_{0},\mathbb{R}^{d})$
\\
(ii) The functions $\mu$ and $\sigma$ are uniformly Lipschitz in $t$, Lipschitz in $x$ and continuously differentiable in $x$.
\\
(iii) The generator $f$ is Lipschitz continuous in $x,y,z$ and $1/2$-Hölder continuous in time, that is: $\exists K>0$ such that for all $(t,x,y,z)$ and $(t',x',y',z')\in [0,T]\times\mathbb{R}^{d}\times\mathbb{R}\times\mathbb{R}^{d},$
\begin{equation*}
|f(t,x,y,z)-f(t',x',y',z')|\leq K(|t-t'|^{1/2}+|x-x'|_{2}+|y-y'|+|z-z'|_{2}).    
\end{equation*}
\\
(iv) $\sup_{t\in[0,T]}|f(t,0,0,0)|<\infty$.
\\
(v) The function $g$ satisfies a linear growth condition.
\end{assumption}
This assumption ensures the existence and uniqueness of an adapted solution $(X,Y,Z)$ to \eqref{e:BSDE}, satisfying 
\begin{equation*}
\mathbb{E}\left[\sup_{0\leq t\leq T}|X_{t}|^{2}+\sup_{0\leq t\leq T}|Y_{t}|^{2}+\int_{0}^{T}|Z_{t}|^{2}\mathrm{d}t\right]< \infty.
\end{equation*}
For the time grid $\pi=\left\{t_{i} | i=0,\dots,N\right\}$, we can introduce the $L^{2}$-regularity of $Z$:
\begin{equation}\label{e:ZRegularity}
\epsilon^{Z}(\pi):=\mathbb{E}\left[\sum_{i=0}^{N-1}\int_{t_{i}}^{t_{i+1}}|Z_{t}-\Bar{Z}_{t_{i}}|_{2}^{2}dt\right],\,\ \text{with} \,\ \Bar{Z}_{t_{i}}:=\frac{1}{\Delta t_{i}}\mathbb{E}_{i}\left[\int_{t_{i}}^{t_{i+1}}Z_{t}dt\right].
\end{equation}

\begin{subsection}{Convergence of the DADM Scheme}
We choose the class of neural network functions $\mathcal{NN}^{\rho}_{d,1,2,m}(\Theta^{\gamma}_{m})$ for the approximation of the BSDE solution $Y_{i}$ and define $(\hat{\mathcal{U}}_{i},\hat{\mathcal{Z}}_{i})$ as the output of the DADM scheme at times $t_{i},i=0,\dots,N.$
\\
We shall define, for $ i=0,\dots,N-1$, 
\begin{equation}
\begin{cases}
V_{i} & =\mathbb{E}_{i}\big[g(X_{N})+f(t_{i},X_{i},V_{i},\Bar{\hat{Z_{i}}})\Delta t_{i}\\
& \qquad \qquad + \sum_{j=i+1}^{N-1}f(t_{j},X_{j},\hat{\mathcal{U}}_{j}(X_{j}),\sigma^{T}(t_{j},X_{j})D_{x}\hat{\mathcal{U}}_{j+1}(X_{j}))\Delta t_{j}\big],\\
\hat{Z}_{i}&=\mathbb{E}_{i}\left[\frac{g(X_{N})\Delta W_{i}}{\Delta t_{i}}+\sum_{j=i+1}^{N-1}f(t_{j},X_{j},\hat{\mathcal{U}}_{j}(X_{j}),\sigma^{T}(t_{j},X_{j})D_{x}\hat{\mathcal{U}}_{j+1}(X_{j}))\frac{\Delta W_{i}\Delta t_{j}}{\Delta t_{i}}\right],\
\end{cases} 
\end{equation}
i.e., an intermediate process that can be best understood as using neural networks to define the process up to step $i+1$ and then a one step of implicit Euler backward scheme.  

We note that by the Markov property of the discretised forward process $X_{i}$, that, for $i=0,\dots,N$ there exist some deterministic functions $v_{i},\hat{z}_{i}$ such that
\begin{equation}\label{e:Markov}
V_{i}=v_{i}(X_{i}),\,\ \Bar{\hat{Z}}_{i}=\hat{z}_{i}(X_{i})\,\ i=0,\dots,N.
\end{equation} For $x\in\mathbb{R}^{d}$, we introduce the flow of the forward Euler scheme:
\begin{equation}\label{e:EulerFlow}
X_{j+1}^{x,i}:=X_{j}^{x}+b(t_{j},X_{j}^{x})\Delta t_{j}+\sigma(t_{j},X_{j}^{x})\Delta W_{j},\,\ j=i,\dots,N-1,\,\ X_{i}^{x}=x.  \end{equation}
This allows us to write the functions defined in \eqref{e:Markov} as
\begin{equation}\label{e:MarkovVZ}
\begin{cases}
v_{i}(x)&=\mathbb{E}[g(X_{N}^{x})+f(t_{i},x,v_{i}(x),\hat{z}_{i}(x))\Delta t_{i}\\
&+\sum_{j=i+1}^{N-1}f(t_{j},X_{j}^{x},\hat{\mathcal{U}}_{j}(X_{j}^{x}),\sigma^{T}(t_{j},X_{j}^{x})D_{x}\hat{\mathcal{U}}_{j+1}(X_{j}^{x}))\Delta t_{j}]\\
\hat{z}_{i}(x)&=\mathbb{E}\left[\left(g(X_{N}^{x})+\sum_{j=i+1}^{N-1}f(t_{j},X_{j}^{x},\hat{\mathcal{U}}_{j}(X_{j}^{x}),\sigma^{T}(t_{j},X_{j}^{x})D_{x}\hat{\mathcal{U}}_{j+1}(X_{j}^{x}))\Delta t_{j}\right)\Delta W_{i}\right].
\end{cases}    
\end{equation}
We note that we may write:
\begin{equation}
v_{i}(x)=\Tilde{v}_{i}(x)+\Delta t_{i}f(t_{i},x,v_{i}(x),\hat{z}_{i}(x))    
\end{equation}
where 
\begin{equation}\label{e:VTildeMarkovDef}
\begin{aligned}
\Tilde{v}_{i}(x)&=\mathbb{E}\left[g(X_{N}^{x})+\sum_{j=i+1}^{N-1}f(t_{j},X_{j}^{x},\hat{\mathcal{U}}_{j}(X_{j}^{x}),\sigma^{T}(t_{j},X_{j}^{x})D_{x}\hat{\mathcal{U}}_{j+1}(X_{j}^{x}))\Delta t_{j}\right]\\
&=\mathbb{E}\left[v^{*}_{i+1}(X_{i+1}^{x})\right]
\end{aligned}
\end{equation}
and \begin{equation}\label{e:V*MarkovDef}
v^{*}_{i}(x):=\mathbb{E}\left[g(X_{N}^{x})+\sum_{j=i}^{N-1}f(t_{j},X_{j}^{x},\hat{\mathcal{U}}_{j}(X_{j}^{x}),\sigma^{T}(t_{j},X_{j}^{x})D_{x}\hat{\mathcal{U}}_{j+1}(X_{j}^{x}))\Delta t_{j}\right].    
\end{equation}

Our main result in this section is an error estimate of the DADM scheme in terms of the $L^{2}$ approximation errors of $v_{i}$ by neural networks $\mathcal{U}_{i}\in\mathcal{N}^{\rho}_{d,1,2,m}(\Theta_{m}^{\gamma}),i=0,\dots,N-1,$ which we define by 
\begin{equation*}
\varepsilon_{i}^{\mathcal{N},m}:=\inf_{\theta\in\Theta^{\gamma}_{m}}\mathbb{E}|v_{i}(X_{i})-\mathcal{U}_{i}(X_{i};\theta)|^{2}
\end{equation*}

\begin{theorem}[Error Approximation of DADM Scheme]
Under Assumption 3, there exists a constant $C>0$, (depending only on $\mu,\sigma,f,g,d,T)$ such that
\begin{equation}\label{e:ApproxErrorDADM}
\begin{aligned}
&\sup_{i\in[[0,N]]}\mathbb{E}|Y_{t_{i}}-\hat{\mathcal{U}}_{i}(X_{i}^{\pi})|^{2}+\mathbb{E}\left[\sum_{i=0}^{N-1}\int_{t_{i}}^{t_{i+1}}|Z_{s}-\sigma^{T}D_{x}\hat{\mathcal{U}}_{i+1}(X_{i}^{\pi})|^{2}_{2}\mathrm{d}s\right]\\
&\qquad \leq C\left(\mathbb{E}|g(X_{T})-g(X^{\pi}_{N})|^{2}+|\pi|+\epsilon^{Z}(\pi)+\frac{\gamma_{m}^{6}}{N}+\sum_{j=0}^{N-1}\varepsilon^{\mathcal{N},m}_{j}\right).
\end{aligned}
\end{equation}
\label{t:ErrorAnalysis}
\end{theorem}
\begin{remark}
 The upper bound in \eqref{e:ApproxErrorDADM} for the approximation error of the DADM scheme consists of five terms. The first three terms come from the time discretization error of the BSDE, as in \cite{bouchardBTZ2004} and \cite{gobet2005regression}, with $(i)$ the strong approximation of the terminal condition, which converges to zero as $|\pi|$ tends to zero, with rate $|\pi|$ when $g$ is Lipschitz, $(ii)$ the strong approximation of the forward Euler scheme and the $L^{2}$ regularity of $Y$, with a rate $|\pi|$ and $(iii)$ the $L^{2}$ regularity of $Z$ as defined in \eqref{e:ZRegularity}. Finally, the last two terms comprise the approximation error by the chosen class of networks. We note that the approximation error term $\sum_{j=0}^{N-1}\varepsilon^{\mathcal{N},m}_{j}$ in \eqref{e:ApproxErrorDADM} is an improvement over the one-step DBDP2 scheme from \cite{hure2019}, with order $\sum_{j=0}^{N-1}N\varepsilon_{j}^{\mathcal{N},m}$.  
\end{remark}
\begin{remark}
The proof of Lemma 5.5, relies on the fact that the terminal function $g$ also satisfies the network error bounds from Lemma \ref{NetworkBoundsLemma} but we cannot of course conclude this \textit{a priori}. One way to circumvent this is to add an extra neural network approximation $\Tilde{g}$ to our scheme, that is $\Tilde{g}\in \mathcal{NN}^{\rho}_{d,1,2,m}(\Theta^{\gamma}_{m})$ is the minimiser of the loss function 
\begin{equation*}
 L_{g}(\mathcal{U})=\mathbb{E}|\mathcal{U}(X_{N})-g(X_{N})|^{2}   
\end{equation*}
which would therefore add an extra error term $\varepsilon^{g}$ to the error approximation result given in \eqref{e:ApproxErrorDADM}. Therefore, by abuse of notation, throughout the rest of the proofs in the chapter we denote by $g$ the terminal function (or network) satisfying the same bounds as the networks in Lemma \ref{NetworkBoundsLemma}.
\end{remark}

Let us first introduce an additional intermediate process $$\hat{V}_{i}=\mathbb{E}_{i}[g(X_{N})+\sum_{j=i}^{N-1}f(t_{j},X_{j},\hat{\mathcal{U}}_{j}(X_{j}),\sigma^{T}(t_{j},X_{j})D_{x}\hat{\mathcal{U}}_{j+1}(X_{j}))\Delta t_{j}],$$
which can be seen as an approximation of the backward process using only the trained neural networks and conditional expectation.

We also introduce the implicit one step backward scheme for the BSDE:
\begin{equation}
\begin{cases}\label{e:BackEulerVBar}
\Bar{V_{i}}&=\mathbb{E}_{i}\left[\Bar{V}_{i+1}+f(t_{i},X_{i},\Bar{V}_{i},\Bar{Z}_{i})\Delta t_{i}\right]\\
\Bar{Z_{i}}&=\mathbb{E}_{i}\left[\Bar{V}_{i+1}\frac{\Delta W_{i}}{\Delta t_{i}}\right],\,\ i=0,\dots,N-1,
\end{cases}
\end{equation}
with $\Bar{V}_{N}=g(X_{N})$\\

To prove Theorem \ref{t:ErrorAnalysis}, we split our error analysis by introducing several intermediate virtual dynamics allowing us to isolate different error components.  The estimation of each error contribution is presented in separate Lemmas. 

We start by considering a term understood as the effect of using neural networks at the ``last-step'' compared with using one step of the Euler scheme, that is, we compare the processes $(V_{i})_{i}$ and $(\hat{V}_{i})_{i}$. We obtain simultaneously a bound for the error incurred by using automatic differentiation as an approximation for $\hat{Z}$. The bounds involve the time step and both the expressivity and the growth controls of our chosen class of neural networks. 

\begin{lemma}\label{ZErrorLemma} \ 

\begin{enumerate}
  \item For $\hat{\mathcal{U}}_{i+1}\in \mathcal{NN}^{\rho}_{d,1,2,m}(\Theta^{\gamma}_{m})$ the solution to the optimisation in $(13)$ we have
  \begin{equation}
    \mathbb{E}|\sigma(t_{i},x)D_{x}\hat{\mathcal{U}}_{i+1}(X_{i})-\hat{Z}_{i}|^{2}\leq \sum_{j=1}^{N-1}\left[\varepsilon^{\mathcal{N},m}_{j}+(\gamma_{m}^{3}+\gamma_{m}^{4})^{2}\Delta t_{j}^{2}\right]\Delta t_{j}  
  \end{equation}
  where $\gamma_{m}$ is a sequence given by \eqref{e:GammaSeq}.
  \\
  \item For $(\hat{V}_{i})_{i}$ and $(V_{i})_{i}$ we have that 
\begin{equation}
\sup_{i\in[[0,N]]}\mathbb{E}|\hat{V}_{i}-{V}_{i}|^{2}\leq C\sum_{i=0}^{N-1}\left(\varepsilon_{i}^{\mathcal{N},m}+(\gamma_{m}^{3}+\gamma_{m}^{4})^{2}\Delta t_{j}^{2}\right)
\end{equation}
\end{enumerate}
 \end{lemma}

\begin{proof}
 We shall apply Taylor's theorem and Ito's lemma to $\mathbb{E}_{i}[\mathcal{U}_{i+1}(X_{i+1})\Delta W_{i}]$ to deduce that there exists a $\xi\in (X_{i},X_{i+1})$ such that 
 \begin{equation}\label{e:DerivExpec}
 \begin{aligned}
  \mathbb{E}_{i}[ & \hat{\mathcal{U}}_{i+1}(X_{i+1})\Delta W_{i}]\\
  &=\mathbb{E}_{i}[\hat{\mathcal{U}}_{i+1}(X_{i})\Delta W_{i}+D_{x}\hat{\mathcal{U}}_{i+1}(X_{i})b(t_{i},X_{i})\Delta t_{i}\Delta W_{i}+\sigma(t_{i},X_{i})D_{x}\hat{\mathcal{U}}_{i+1}(X_{i})\Delta W_{i}^{2}\\
  &\qquad  + \frac{1}{2}D^{2}_{x}\hat{\mathcal{U}}_{i+1}(X_{i})b^{2}(t_{i},X_{i})\Delta t_{i}^{2}\Delta W_{i}+\frac{1}{2}D^{2}_{x}\hat{\mathcal{U}}_{i+1}(X_{i})\sigma^{2}(t_{i},X_{i})\Delta W_{i}^{3}\\
  &\qquad+D^{2}_{x}\hat{\mathcal{U}}_{i+1}(X_{i})b(t_{i},X_{i})\sigma(t_{i},X_{i})\Delta t_{i}\Delta W_{i}^{2}+\frac{1}{6}D_{x}^{3}\hat{\mathcal{U}}_{i+1}(\xi)b^{2}(t_{i},X_{i})\Delta t_{i}^{3}\Delta W_{i}\\
  &\qquad+\frac{1}{2}D_{x}^{3}\hat{\mathcal{U}}_{i+1}(\xi)\left(b^{2}(t_{i},X_{i})\sigma(t_{i},X_{i})\Delta t_{i}^{2}\Delta W_{i}^{2}+b(t_{i},X_{i})\sigma^{2}(t_{i},X_{i})\Delta t_{i}\Delta W_{i}^{3}\right)\\
  &\qquad+\frac{1}{6}D_{x}^{3}\hat{\mathcal{U}}_{i+1}(\xi)\sigma^{3}(t_{i},X_{i})\Delta W_{i}^{4}].
  \end{aligned}
 \end{equation}
 Simplifying equation \eqref{e:DerivExpec} leads to 
 \begin{equation}
 \frac{1}{\Delta t_{i}}\mathbb{E}_{i}[\hat{\mathcal{U}}_{i+1}(X_{i+1})\Delta W_{i}]=\sigma(t_{i},X_{i})D_{x}\hat{\mathcal{U}}_{i+1}(X_{i})+P(X_{i})\Delta t_{i}
\end{equation}
 where $P(X_{i})=D^{2}_{x}\hat{\mathcal{U}}_{i+1}(X_{i})b(t_{i},X_{i})\sigma(t_{i},X_{i})+\frac{1}{6}D^{3}_{x}\hat{\mathcal{U}}_{i+1}(\xi)\sigma(t_{i},X_{i})^{3}$.
\\
We have that the coefficients $\sigma$ and $b$ are Lipschitz and that the network $\hat{\mathcal{U}}_{i+1}$ satisfies Lemma 3.3 so we have that 
$$P(X_{i})\leq C(\gamma_{m}^{3}+\gamma_{m}^{4})$$ for some constant $C$.
\\
We use this to study the error term corresponding to our approximation of $Z_{t}$, that is $\Delta t_{i}\mathbb{E}|\hat{Z}_{i}-\sigma(t_{i},X_{i})D_{x}\hat{\mathcal{U}}_{i+1}(X_{i})|^{2}$.
\\

Firstly, we introduce the process $Z_{i}=\frac{1}{\Delta t_{i}}\mathbb{E}_{i}\left[V_{i+1}\Delta W_{i}\right]$ and we note that $$\hat{Z}_{i}=\mathbb{E}_{i}\left[\hat{V}_{i+1}\frac{\Delta W_{i}}{\Delta t_{i}}\right]$$

Now,
\begin{equation}
\begin{aligned}
\mathbb{E}|V_{i}-\hat{V}_{i}|^{2}&=\Delta t_{i}^{2}\mathbb{E}|f(t_{i},X^{\pi}_{i},V_{i},\hat{Z}_{i})-f(t_{i},X^{\pi}_{i},\hat{\mathcal{U}}_{i}(X^{\pi}_{i}),\sigma^{T}(t_{i},X_{i}^{\pi})D_{x}\hat{\mathcal{U}}_{i+1}(X_{i}^{\pi}))|^{2}\\
&\leq 2K\Delta t_{i}^{2}\left(\mathbb{E}|V_{i}-\hat{\mathcal{U}}_{i}(X_{i}^{\pi})|^{2}+\mathbb{E}|\hat{Z}_{i}-\sigma^{T}(t_{i},X_{i}^{\pi})D_{x}\hat{\mathcal{U}}_{i+1}(X_{i}^{\pi})|^{2}\right)\\
&\leq C\Delta t_{i}^{2}(\varepsilon_{i}+\Delta t_{i}\mathbb{E}|\hat{Z}_{i}-\sigma^{T}(t_{i},X_{i}^{\pi})D_{x}\hat{\mathcal{U}}_{i+1}(X_{i}^{\pi})|^{2}.
\end{aligned}
\end{equation}
Moreover, we have that:
\begin{equation}\label{e:ZNetError}
\begin{aligned}
\Delta t_{i}\mathbb{E}|\hat{Z}_{i}-\sigma(t_{i},X_{i})D_{x}\hat{\mathcal{U}}_{i+1}(X_{i})|^{2}&\leq C\Delta t_{i}\mathbb{E}|Z_{i}-\hat{Z}_{i}|^{2}+C\Delta t_{i}\mathbb{E}|Z_{i}-\sigma(t_{i},X_{i})D_{x}\hat{\mathcal{U}}_{i+1}(X_{i})|^{2}\\
\leq & C\Delta t_{i}\mathbb{E}\left|\frac{1}{\Delta t_{i}}\mathbb{E}_{i}[V_{i+1}\Delta W_{i}]-\frac{1}{\Delta t_{i}}\mathbb{E}_{i}[\hat{V}_{i+1}\Delta W_{i}]\right|^{2}\\
+& C\Delta t_{i}\mathbb{E}\left|\frac{1}{\Delta t_{i}}\mathbb{E}_{i}[V_{i+1}\Delta W_{i}]-\frac{1}{\Delta t_{i}}\mathbb{E}_{i}[\hat{\mathcal{U}}_{i+1}(X_{i+1})\Delta W_{i}]\right|^{2}\\
+& C\Delta t_{i}(\gamma_{m}^{3}+\gamma_{m}^{4})^{2}\Delta t_{i}^{2}
\end{aligned}
\end{equation}
Hence, applying the discrete Gronwall inequality leads to 
\begin{equation}\label{e:FinalZ}
\sup_{i=0,\dots,N-1}\mathbb{E}|V_{i}-\hat{V}_{i}|^{2}\leq C \sum_{j=1}^{N-1}\left[\varepsilon_{j}^{\mathcal{N},m}+\varepsilon_{j+1}^{\mathcal{N},m}+(\gamma_{m}^{3}+\gamma_{m}^{4})^{2}\Delta t_{j}^{2}\right]\Delta t_{j}    
\end{equation}
substituting \eqref{e:FinalZ} into \eqref{e:ZNetError} gives the desired result.
\end{proof}

The following Lemma provides a bound for the error incurred by using $\hat{\mathcal{U}}_{i}$ as an approximator for the scheme $(V_{i})_{i}$
\begin{lemma}
    For $({V}_{i})_{i}$ and $\hat{\mathcal{U}}_{i}$, we have
    \begin{equation}
 \mathbb{E}|V_{i}-\hat{\mathcal{U}}_{i}(X_{i})|^{2}+\Delta t_{i}\mathbb{E}|\hat{Z}_{i}-\sigma(t_{i},X_{i})D_{x}\hat{\mathcal{U}}_{i+1}(X_{i})|^{2}\leq C\sum_{j=1}^{N-1}\varepsilon_{j}^{\mathcal{N},m}+(\gamma_{m}^{3}+\gamma_{m}^{4})^{2}\Delta t_{j}^{2}.  
\end{equation}
\end{lemma}

\begin{proof}
    We apply the martingale representation theorem to $V_{i}$ to determine the existence of an adapted process $\{\hat{Z}^{M}_{s}|0\leq s\leq T\}$ such that:
 
 \begin{equation}
 \begin{split}
        g(X_{N}) & + f(t_{i},X_{i},V_{i},\hat{Z}_{i})\Delta t_{i} \\
        & +\sum_{j=i+1}^{N-1}f(t_{j},X_{j},\hat{\mathcal{U}}_{j}(X_{j}),\sigma^{T}(t_{j},X_{j})D_{x}\hat{\mathcal{U}}_{j+1}(X_{j}))\Delta t_{j}=V_{i}+\int_{t_{i}}^{t_{N}}\hat{Z}^{M}_{s}dW_{s}.
 \end{split}
 \end{equation}
 We note that:
 \begin{equation}
 \hat{Z}_{i}=\frac{1}{\Delta t_{i}}\mathbb{E}_{i}\left[\int_{t_{i}}^{t_{i+1}}\hat{Z}_{s}^{M}\mathrm{d}s\right] \implies \mathbb{E}_{i}\left[\int_{t_{i}}^{t_{i+1}}(\hat{Z}_{i}-\hat{Z}^{M}_{s})ds\right]=0.
 \end{equation}
 \\
 Inserting into the loss function for our scheme gives:
 
 \begin{align*}
 L_{i}(\mathcal{U}_{i})&=\mathbb{E}|V_{i}-\mathcal{U}_{i}(X_
 {i})+\Delta t_{i}[f(t_{i},X_{i},V_{i},\hat{Z}_{i})-f(t_{i},X_{i},\mathcal{U}_{i}(X_{i}),\sigma^{T}(t_{j},X_{j})D_{x}\mathcal{U}_{i+1}(X_{i})\\
 +&\sum_{j=i+1}^{N-1}\int_{t_{j}}^{t_{j+1}}|\hat{Z}_{s}-\sigma^{T}(t_{j},X_{j})D_{x}\hat{\mathcal{U}}_{j+1}(X_{j})|dW_{s}+\int_{t_{i}}^{t_{i+1}}[\hat{Z}_{s}-\sigma^{T}(t_{i},X_{i})D_{x}\mathcal{U}_{i+1}(X_{i})]dW_{s}|^{2}\\
 &=\Tilde{L}(\mathcal{U}_{i})+\mathbb{E}\left[\sum_{j=i}^{N-1}\int_{t_{j}}^{t_{j+1}}|\hat{Z}_{s}-\hat{Z}_{j}|^{2} ds\right]+\sum_{j=i+1}^{N-1}\Delta t_{j}\mathbb{E}|\hat{Z}_{j}-\sigma^{T}(t_{j},X_{j})D_{x}\hat{\mathcal{U}}_{j+1}(X_{j})|^{2}
 \end{align*}
 where
 \begin{align*}
 \Tilde{L}(\mathcal{U}_{i})=& \mathbb{E}\Big|V_{i}  -\mathcal{U}_{i}(X_{i})+\Delta t_{i}\left[f(t_{i},X_{i},V_{i},\hat{Z}_{i})-f(t_{i},X_{i},\mathcal{U}_{i}(X_{i}),\sigma^{T}(t_{i},X_{i})D_{x}\mathcal{U}_{i+1}(X_{i})\right]\Big|^{2}\\
 & +\Delta t_{i}\mathbb{E}|\hat{Z}_{i}-\sigma^{T}(t_{i},X_{i})D_{x}\mathcal{U}_{i+1}(X_{i})|^{2}. 
 \end{align*}
 \\
 By utilising the Lipschitz continuity of $f$ we can see that:
 \begin{equation}
\Tilde{L}(\mathcal{U}_{i})\leq C\left(\mathbb{E}|V_{i}-\mathcal{U}_{i}(X_{i})|^{2}+\Delta t_{i}\mathbb{E}\left|\hat{Z}_{i}-\sigma^{T}(t_{i},X_{i})D_{x}\mathcal{U}_{i+1}(X_{i})\right|^{2}\right).     
 \end{equation}
 \\
 Furthermore, we can use Young's inequality with $\beta\in(0,1)$ to determine that:
 \begin{align*}
\Tilde{L}_{i}(\mathcal{U}_{i})&\geq (1-\beta)\mathbb{E}|V_{i}-\mathcal{U}_{i}(X_{i})|^{2}+\Delta t_{i}\mathbb{E}|\hat{Z}_{i}-\sigma^{T}(t_{i},X_{i})D_{x}\hat{\mathcal{U}}_{i+1}(X_{i})|^{2}\\
&\quad +(1-\frac{1}{\beta})|\Delta t_{i}|^{2}\mathbb{E}|f(t_{i},X_{i},\mathcal{U}_{i}(X_{i}),\sigma^{T}(t_{i},X_{i})D_{x}\mathcal{U}_{i+1}(X_{i}))-f(t_{i},X_{i},V_{i},\hat{Z}_{i})|^{2}\\
&\geq (1-\beta)\mathbb{E}|V_{i}-\mathcal{U}_{i}(X_{i})|^{2}+\Delta t_{i}\mathbb{E}|\hat{Z}_{i}-\sigma^{T}(t_{i},X_{i})D_{x}\mathcal{U}_{i+1}(X_{i})|^{2}\\
&\quad -2\frac{K^{2}}{\beta}|\Delta t_{i}|^{2}(\mathbb{E}|\mathcal{U}_{i}(X_{i})-V_{i}|^{2}+\mathbb{E}|\sigma^{T}(t_{i},X_{i})D_{x}\mathcal{U}_{i+1}(X_{i})-\hat{Z}_{i}|^{2})\\
&\geq (1-(4K^{2}+\frac{1}{2})\Delta t_{i})\mathbb{E}|V_{i}-\mathcal{U}_{i}(X_{i})|^{2}+\frac{1}{2}\Delta t_{i}\mathbb{E}|\hat{Z}_{i}-\sigma^{T}(t_{i},X_{i})D_{x}\mathcal{U}_{i+1}(X_{i})|^{2}
\end{align*}
where we have used the Lipschitz continuity of $f$ and set $\beta=4K^{2}\Delta t_{i}$. We apply the above inequalities to $\mathcal{U}=\hat{\mathcal{U}}_{i}$, which is a minimiser of $\Tilde{L}(\hat{\mathcal{U}_{i}})$ and obtain, for $\Delta t_{i}$ sufficiently small:
\begin{align*}
\mathbb{E}|V_{i}-\hat{\mathcal{U}}_{i}(X_{i})|^{2} & +\Delta t_{i}\mathbb{E}|\hat{Z}_{i}-\sigma^{T}(t_{i},X_{i})D_{x}\hat{\mathcal{U}}_{i+1}(X_{i})|^{2}\\
\leq & C\left(\mathbb{E}|V_{i}-\mathcal{U}_{i}(X_{i})|^{2}+\Delta t_{i}\mathbb{E}|\hat{Z}_{i}-\sigma^{T}(t_{i},X_{i})D_{x}\mathcal{U}_{i+1}(X_{i})|^{2}\right). 
\end{align*}
We use Lemma \ref{ZErrorLemma} and the definition of $\varepsilon_{i}^{\mathcal{N},m}$ to write
\begin{equation}\label{e:NetError}
 \mathbb{E}|V_{i}-\hat{\mathcal{U}}_{i}(X_{i})|^{2}+\Delta t_{i}\mathbb{E}|\hat{Z}_{i}-\sigma(t_{i},X_{i})D_{x}\hat{\mathcal{U}}_{i+1}(X_{i})|^{2}\leq C\sum_{j=1}^{N-1}\varepsilon_{j}^{\mathcal{N},m}+\varepsilon^{\mathcal{N},m}_{j+1}+(\gamma_{m}^{3}+\gamma_{m}^{4})^{2}\Delta t_{j}^{2}.  
\end{equation}
\end{proof}

The following Lemma illustrates a bound for the error between the classical one step scheme for the BSDE, $(\Bar{V}_{i})_i$, and our multistep scheme given by $(\hat{V}_{i})_i$.

\begin{lemma}
For $(\hat{V}_{i})_{i}$ and $(\Bar{V}_{i})_{i}$ we have
\begin{equation}\label{e:GronY}
  \sup_{i\in[[0,N]]}\mathbb{E}|\Bar{V}_{i}-\hat{V}_{i}|^{2}\leq C\sum_{i=0}^{N-1}\left(\varepsilon_{i}^{\mathcal{N},m}+(\gamma_{m}^{3}+\gamma_{m}^{4})^{2}\Delta t_{j}^{2}\right).   
 \end{equation}
 \end{lemma}

 \begin{proof}
     We use the recursive definitions of $(\Bar{V}_{i})_{i}$ and $(\hat{V}_{i})_{i}$ in conjunction with Young's inequality, the Lipschitz continuity of $f$ and the Cauchy-Schwarz inequality to write
\begin{equation}\label{e:VDiff}
\begin{aligned}
\mathbb{E}|\Bar{V}_{i}-\hat{V}_{i}|^{2}&\leq (1+\beta)\mathbb{E}\left|\mathbb{E}_{i}[\Bar{V}_{i+1}-\hat{V}_{i+1}]\right|^{2}\\
&+2K^{2}\left(1+\frac{1}{\beta}\right)|\Delta t_{i}|^{2}\left(\mathbb{E}|\Bar{V}_{i}-\hat{\mathcal{U}}_{i}(X^{\pi}_{i})|^{2}+\mathbb{E}|\Bar{Z}_{i}-\sigma^{T}(t_{i},X^{\pi}_{i})D_{x}\hat{\mathcal{U}}_{i+1}|^{2}\right)\\
&\leq (1+\beta)\mathbb{E}\left|\mathbb{E}_{i}[\Bar{V}_{i+1}-\hat{V}_{i+1}]\right|^{2}+2K^{2}(1+\frac{1}{\beta})|\Delta t_{i}|^{2}(3\mathbb{E}|\Bar{V}_{i}-\hat{V}_{i}|^{2}+2\mathbb{E}|\Bar{Z}_{i}-\hat{Z}_{i}|^{2})\\
&+2K^{2}(1+\frac{1}{\beta})|\Delta t_{i}|^{2}\left(3\mathbb{E}|\hat{V}_{i}-V_{i}|^{2}+3\mathbb{E}|V_{i}-\hat{\mathcal{U}}_{i}|^{2}+2\mathbb{E}|\hat{Z}_{i}-\sigma^{T}D_{x}\hat{\mathcal{U}}_{i+1}|^{2}\right)\\
&\leq (1+\beta)\mathbb{E}\left|\mathbb{E}_{i}[\Bar{V}_{i+1}-\hat{V}_{i+1}]\right|^{2}+(1+\beta)\frac{2K^{2}|\Delta t_{i}|^{2}}{\beta}(3\mathbb{E}|\Bar{V}_{i}-\hat{V}_{i}|^{2}+2\mathbb{E}|\Bar{Z}_{i}-\hat{Z}_{i}|^{2})\\
&+CK^{2}(1+\frac{1}{\beta})\Delta t_{i}C\sum_{j=1}^{N-1}\varepsilon_{j}^{\mathcal{N},m}+\varepsilon^{\mathcal{N},m}_{j+1}+(\gamma_{m}^{3}+\gamma_{m}^{4})^{2}\Delta t_{j}^{2},
\end{aligned}    
\end{equation}
where we use \eqref{e:ZNetError} and \eqref{e:NetError} in the last inequality. Then, by \eqref{e:BackEulerVBar} and \eqref{e:VZ} we have
\begin{equation*}
\begin{aligned}
    \Delta t_{i}(Z_{i}-\hat{Z}_{i})&=\mathbb{E}_{i}\left[\Delta W_{i}(\Bar{V}_{i+1}-\hat{V}_{i+1})\right]\\
    &=\mathbb{E}_{i}\left[\Delta W_{i}\left(\Bar{V}_{i+1}-\hat{V}_{i+1}-\mathbb{E}_{i}[\Bar{V}_{i+1}-\hat{V}_{i+1}]\right)\right],
\end{aligned}
\end{equation*}
and by Cauchy-Schwarz inequality
\begin{equation}\label{e:CSZ}
\Delta t_{i}\mathbb{E}|Z_{i}-\hat{Z}_{i}|^{2}_{2}\leq d\left(\mathbb{E}|\Bar{V}_{i+1}-\hat{V}_{i+1}|^{2}-\mathbb{E}\left|\mathbb{E}_{i}[\Bar{V}_{i+1}-\hat{V}_{i+1}]\right|^{2}\right).    
\end{equation}
We insert \eqref{e:CSZ} into \eqref{e:VDiff} and set $\beta=4dK^{2}\Delta t_{i}$ to obtain
\begin{equation*}
\begin{split}
    (1-C\Delta t_{i})\mathbb{E}|\Bar{V}_{i}-\hat{V}_{i}|^{2}\leq (1 & +C\Delta t_{i})\mathbb{E}|\Bar{V}_{i+1}-\hat{V}_{i+1}|^{2} \\
    & +(1+C\Delta t_{i})\sum_{j=1}^{N-1}\left(\varepsilon_{j}^{\mathcal{N},m}+\varepsilon^{\mathcal{N},m}_{j+1}+(\gamma_{m}^{3}+\gamma_{m}^{4})^{2}\Delta t_{j}^{2}\right).
\end{split}
  \end{equation*}
 We apply the discrete Gronwall lemma and recall that $\Bar{V}_{N}=\hat{V}_{N}=g(X_{N})$ to obtain
 \begin{equation}\label{e:GronY}
  \sup_{i\in[[0,N]]}\mathbb{E}|\Bar{V}_{i}-\hat{V}_{i}|^{2}\leq C\sum_{i=0}^{N-1}\left(\varepsilon_{i}^{\mathcal{N},m}+\varepsilon^{\mathcal{N},m}_{i+1}+(\gamma_{m}^{3}+\gamma_{m}^{4})^{2}\Delta t_{j}^{2}\right).   
 \end{equation}
 \end{proof}

We now begin the proof of Theorem \ref{t:ErrorAnalysis}:

\begin{proof}
We recall from \cite{gobet2005regression} the time discretization error
\begin{equation}\label{e:TimeDisc}
 \sup_{i\in[[0,N]]}\mathbb{E}|Y_{t_{i}}-\Bar{V}_{i}|^{2}+\mathbb{E}\left[\sum_{i=0}^{N-1}\int_{t_{i}}^{t_{i+1}}|Z_{s}-\Bar{Z}_{i}|^{2}_{2}\mathrm{d}s\right]\leq C\left(\mathbb{E}|g(X_{T})-g(X^{\pi}_{N})|^{2}+|\pi|+\epsilon^{Z}(\pi)\right)   \end{equation}
 for some constant $C$ depending only on the coefficients satisfying Assumption \ref{LipschitzAssump}.

We also recall the intermediate process $$\hat{V}_{i}=\mathbb{E}_{i}[g(X_{N})+\sum_{j=i}^{N-1}f(t_{j},X_{j},\hat{\mathcal{U}}_{j}(X_{j}),\sigma^{T}(t_{j},X_{j})D_{x}\hat{\mathcal{U}}_{j+1}(X_{j}))\Delta t_{j}].$$

Using the tower property of conditional expectation allows us to write:
\begin{equation}
    \hat{V}_{i}=\mathbb{E}_{i}\left[\hat{V}_{i+1}+f(t_{i},X_{i},\hat{\mathcal{U}}_{i}(X_{i}),\sigma^{T}(t_{i},X_{i})D_{x}\hat{\mathcal{U}}_{i+1}(X_{i}))\Delta t_{i}\right].
    \end{equation}

We also remark that:
\begin{equation}\label{e:VZ}
\hat{Z_{i}}=\mathbb{E}_{i}\left[\hat{V}_{i+1}\frac{\Delta W_{i}}{\Delta t_{i}}\right].
\end{equation}
We decompose the $Y$ error as:
\begin{equation}\label{e:YDecomp}
 \mathbb{E}_{i}|Y_{t_{i}}-\hat{\mathcal{U}}_{i}|^{2}\leq 4(\mathbb{E}|Y_{t_{i}}-\Bar{V}_{i}|^{2}+\mathbb{E}|\Bar{V}_{i}-\hat{V}_{i}|^{2}+\mathbb{E}|\hat{V}_{i}-V_{i}|^{2}+\mathbb{E}|V_{i}-\hat{\mathcal{U}}_{i}(X_{i})|^{2}):=4(T^{1}_{i}+T^{2}_{i}+T^{3}_{i}+T^{4}_{i}).
 \end{equation}
 The first term, $T_{i}^{1}$, is the time discretization error that has been well studied by previous authors and so we dedicate our analysis to the remaining three terms.

Applying Lemma 5.6, Lemma 5.7 and using \eqref{e:TimeDisc} immediately leads to the required result for the $Y$ error.
 \\
 
 Finally, we decompose the approximation error for the $Z$ component
 \begin{equation}\label{e:ZDecomp}
 \begin{aligned}
 &\mathbb{E}\left[\sum_{i=0}^{N-1}\int_{t_{i}}^{t_{i+1}}|Z_{s}-\sigma^{T}(t_{i},X^{\pi}_{i})D_{x}\hat{\mathcal{U}}_{i+1}(X_{i}^{\pi})|^{2}_{2}\mathrm{d}s\right]\\
 &\leq 3\sum_{i=0}^{N-1}\left(\mathbb{E}\left[\int_{t_{i}}^{t_{i+1}}|Z_{s}-\Bar{Z}_{i}|^{2}_{2}\mathrm{d}s\right]+\Delta t_{i}\mathbb{E}|\Bar{Z}_{i}-\hat{Z}_{i}|^{2}_{2}+\Delta t_{i}\mathbb{E}|\hat{Z}_{i}-\sigma^{T}(t_{i},X_{i}^{\pi})D_{x}\hat{\mathcal{U}}_{i+1}(X_{i}^{\pi})|^{2}_{2}\right).
 \end{aligned}
 \end{equation}
 We sum the inequality \eqref{e:CSZ} and use \eqref{e:VDiff} in order to obtain 
 \begin{equation}\label{e:ZHatDiff}
 \begin{aligned}
 &\sum_{i=0}^{N-1}\Delta t_{i}\mathbb{E}|\Bar{Z}_{i}-\hat{Z}_{i}|^{2}\\
 &\leq d\sum_{i=0}^{N-1}\left(\mathbb{E}|\Bar{V}_{i}-\hat{V}_{i}|^{2}-\mathbb{E}|\mathbb{E}_{i}\left[\Bar{V}_{i+1}-\hat{V}_{i+1}\right]|^{2}\right)\\
 &\leq d\sum_{i=0}^{N-1}\left(\beta \mathbb{E}|\mathbb{E}_{i}\left[\Bar{V}_{i+1}-\hat{V}_{i+1}\right]|^{2}+(1+\frac{1}{\beta})(2K^{2}|\Delta t_{i}|^{2})(3\mathbb{E}|\Bar{V}_{i}-\hat{V}_{i}|^{2}+2\mathbb{E}|\Bar{Z}_{i}-\hat{Z}_{i}|^{2} \right.\\
 &\qquad +CK^{2}(1+\frac{1}{\beta})\Delta t_{i}(\varepsilon_{i}^{\mathcal{N},m}+\varepsilon^{\mathcal{N},m}_{i+1}+(\gamma_{m}^{3}+\gamma_{m}^{4})^{2}\Delta t_{j}^{2}))\\
 &\leq  \frac{1}{2}\sum_{i=0}^{N-1}\Delta t_{i}\mathbb{E}|\Bar{Z}_{i}-\hat{Z}_{i}|^{2}_{2} + d\sum_{i=0}^{N-1}\Big(\frac{8dK^{2}\Delta t_{i}}{1-8dK^{2}\Delta t_{i}}\mathbb{E}|\mathbb{E}_{i}[\Bar{V}_{i+1}-\hat{V}_{i+1}]|^{2}+\frac{3}{4d}\Delta t_{i}\mathbb{E}|\Bar{V}_{i}-\hat{V}_{i}|^{2}\\
 & \qquad\qquad\qquad\qquad\qquad\qquad\qquad+ \frac{C}{8d}(\varepsilon_{i}^{\mathcal{N},m}+\varepsilon^{\mathcal{N},m}_{i+1}+(\gamma_{m}^{3}+\gamma_{m}^{4})^{2}\Delta t_{j}^{2})\Big) ,
 \end{aligned}
 \end{equation}
 where we have chosen $\beta=\frac{8dK^{2}\Delta t_{i}}{1-8dK^{2}\Delta t_{i}}=O(\Delta t_{i})$ for $\Delta t_{i}$ small enough and we have applied Lemma 5.5, Lemma 5.6 and Lemma 5.7. 
 Substituting \eqref{e:TimeDisc}, and \eqref{e:ZHatDiff} into \eqref{e:ZDecomp} proves the required bound for the approximation error of $Z$ and thus completes the proof.
 \end{proof}

\end{subsection}
\end{section}

\begin{section}{Numerical Examples}
In the following numerical implementation section we shall use networks with a single hidden layer $(L=1)$ and a ReLU activation function. The learning rate ($\alpha$) will be initialised at $\alpha=0.01$ and we implement a learning rate scheduler with loss tolerance of $0.01$ before a halving of the learning rate. We use $5000$ iterations of stochastic gradient descent with batch size $M=1000$.
\begin{subsection}{Bounded example}
We consider the following decoupled FBSDE
\begin{equation}\label{e:Bounded}
\begin{cases}
dX_{t}&=\mu dt +\sigma dW_{t},\,\ X_{0}=x_{0},\\
-dY_{t}&=((\cos(\Bar{X})+0.2\sin(\Bar{X}))\exp({\frac{T-t}{2}}
)\\
&-\frac{1}{2}\left(\sin(\Bar{X})\cos(\Bar{X})\exp{(T-t)}\right)^{2}+\frac{1}{2}(Y_{t}\Bar{Z})^{2})dt-Z_{t}dW_{t}, \\
Y_{T}&=\cos(\Bar{X})
\end{cases}
\end{equation}
where $\Bar{X}=\sum_{i=1}^{d}X^{i}_{t}$ and $\Bar{Z}=\sum_{i=1}^{d}Z^{i}_{t}$.
\\

By applying the nonlinear Feynman-Kac formula we determine that $Y_{t}=u(t,X_{t})$ where $u$ is the solution to the PDE \eqref{e:PDE} with
$$f(t,x,y,z)=-\left(\cos(\Bar{x})+0.2\sin(\Bar{x}))e^{\frac{T-t}{2}}\right)+\frac{1}{2}\left(\sin(\Bar{x})\cos(\Bar{x})\exp{(T-t)}\right)^{2}-\frac{1}{2}(yz)^{2}.$$
The true solution is given by 
\begin{equation}
u(x,t)=\exp\left(\frac{T-t}{2}\right)\cos(\Bar{x}).
\end{equation}

We take $d=1$, $T=2, \mu=0.2, \sigma=1$ and $x_{0}=1$. We note that for this value of the maturity $T$ and number of time steps, $N=180$, the Deep BSDE scheme does not converge. The following plots below are for the solution at given time steps
\\

\begin{table}[H]
    \centering
\begin{tabular}{ |p{3cm}||p{3cm}|p{3cm}|p{3cm}|  }
 \hline
 \multicolumn{4}{|c|}{Scheme errors} \\
 \hline
 & Averaged value & Standard deviation &Relative error (\%)\\
 \hline
  (DBDP2)   & 1.4364    & \textbf{0.0140} &   2.19\\
  (DBSDE)&   NotConv& NotConv   &NotConv\\
 (DADM) &1.4594 & 0.022&  \textbf{0.63}\\
 
 \hline
\end{tabular}
\caption{Estimate of $u(0,x_{0})$ in the case \eqref{e:Bounded}, where $d=1,\,\ x_{0}=1,\,\ T=2$ with $N=180$ time steps. The average and standard deviation are computed over 5 independent runs of the schemes.}
\label{Figure 1}
\end{table}

\begin{figure}[H]
\includegraphics[scale=0.8]{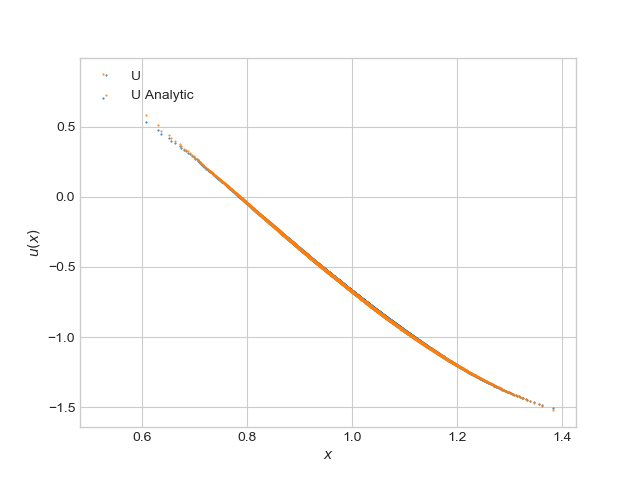}
\caption{Estimate of $u$ using DADM Scheme in case \eqref{e:Bounded} with $T=2$ and $N=180$ at $t=0.0006$.}
\label{Figure 2}
\end{figure}

We now increase the dimension of the state variable and set the maturity to $T=1$.

\begin{table}[H]
\centering
\begin{tabular}{ |p{3cm}||p{3cm}|p{3cm}|p{3cm}|  }
 \hline
 \multicolumn{4}{|c|}{Scheme errors} \\
 \hline
 & Averaged value & Standard deviation &Relative error (\%)\\
 \hline
  (DBDP2)   & -1.3941    & \textbf{0.0010} &   0.77\\
  (DBSDE)&   -1.3891& 0.0021   &0.41\\
 (DADM) &-1.3899 & 0.0012&  \textbf{0.47}\\
 \hline
\end{tabular}
\caption{Estimate of $u(0,x_{0})$ in the case \eqref{e:Bounded}, where $d=10, x_{0}=1_{10}, T=1$ with $N=60$ time steps.The average and standard deviation are computed over 5 independent runs of the schemes.}
\label{Figure 3}
\end{table}

\begin{table}[H]
    \centering
\begin{tabular}{ |p{3cm}||p{3cm}|p{3cm}|p{3cm}|  }
 \hline
 \multicolumn{4}{|c|}{Scheme errors} \\
 \hline
 & Averaged value & Standard deviation &Relative error (\%)\\
 \hline
  (DBDP2)   & 0.6698    &0.0062 &   0.45\\
  (DBSDE)&   0.6805& 0.0021   &1.14\\
 (DADM) &0.6712 & \textbf{0.0010}& \textbf{0.27}  \\
 \hline
\end{tabular}

\caption{Estimate of $u(0,x_{0})$ in the case \eqref{e:Bounded}, where $d=20, x_{0}=1, T=1$ with $N=60$ time steps. The average and standard deviation are computed over 5 independent runs of the schemes.}

\label{Figure 4}
\end{table}
\end{subsection}

\begin{subsection}{Unbounded example}
We consider the parameters 
\\
\begin{equation}\label{e:Unbounded}
\begin{cases}
\mu=0,\sigma=\frac{I_{d}}{\sqrt{d}}\\
f(x,y,z)=k(x)+\frac{y}{\sqrt{d}}(1_{d}.z)+\frac{y^{2}}{2}
\end{cases} 
\end{equation}
such that the solution of the PDE is given by 

$$u(t,x)=\frac{T-t}{d}\sum_{i=1}^{d}\left(\sin(x_{i})1_{x_{i}<0}+x_{i}1_{x_{i}>0}\right)+\cos\left(\sum_{i=1}^{d}ix_{i}\right).$$
We shall begin with tests in dimension $d=1$ and we compare the Deep BSDE scheme, DBDP2 scheme and our DADM scheme. The results are listed in Figure 5 and we show the plots in dimension 1 for the DADM Scheme below
\\

\begin{table}[H]
\begin{tabular}{ |p{3cm}||p{3cm}|p{3cm}|p{3cm}|  }
 \hline
 \multicolumn{4}{|c|}{Scheme errors} \\
 \hline
 & Averaged value & Standard deviation &Relative error (\%)\\
 \hline
  (DBDP2)   & 1.3719    & \textbf{0.0040} &   0.41\\
  (DBSDE)&   1.3708  & 0.0031   &0.49\\
 (DADM) &1.3735 & 0.0041&  \textbf{0.29}\\
 
 \hline
\end{tabular}
\label{Figure 5}
\caption{Estimates of $u(0,x_{0})$ for \eqref{e:Unbounded},with $d=1,x_{0}=0.5$ with 60 time steps. The average and standard deviation are computed over 5 independent runs of the algorithm. The theoretical solution is 1.3776.}
\end{table}

\begin{figure}[H]
\includegraphics[scale=0.8]{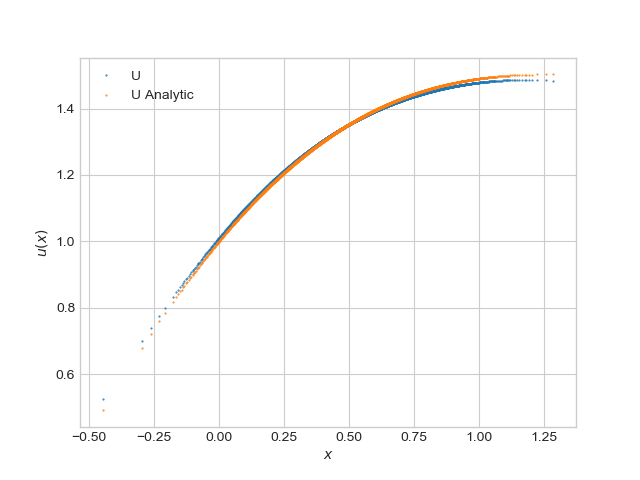}
\label{Figure 6}
\caption{Estimate of $u$ using DADM Scheme in the case \eqref{e:Unbounded}. We take $d=1,x_{0}=0.5$ and $t=0.01667$.}
\end{figure}

We now increase the dimension to $d=8$ in order to show the merits of our method compared to classical methods such as finite differences that struggle to converge above dimension $d=4$. We note that the accuracy of the scheme is not as good as in the previous example, but we notice that the DADM scheme provides improved performance compared to the DBDP2 scheme and DBSDE scheme, which does not converge.

\begin{table}[H]
    \centering    
\begin{tabular}{ |p{3cm}||p{3cm}|p{3cm}|p{3cm}|  }
 \hline
 \multicolumn{4}{|c|}{Scheme errors} \\
 \hline
 & Averaged value & Standard deviation &Relative error\\
 \hline
  (DBDP2)   & 1.0327    & \textbf{0.0081} &   11.02\\
  (DBSDE)&   NotConv  & NotConv   &NotConv\\
 (DADM) &1.0675 & 0.01241&  \textbf{8.01}\\
 
 \hline
\end{tabular}
\label{Figure 7}
\caption{Estimates of $u(0,x_{0})$ for \eqref{e:Unbounded},with $d=8,x_{0}=(0.5)\mathbf{1}_{8}$ with 60 time steps. The average and standard deviation are computed over 5 independent runs of the algorithm. The theoretical solution is 1.1603.}
\end{table}

\end{subsection}

\begin{subsection}{Remarks on the numerical results}
We can see that,from our example $(6.1)$
,when the maturity is large our the DADM scheme outperforms the benchmark DBSDE as that method fails to converge due to the high number of times steps in the discretisation, We also note that our scheme outperforms its one step counterpart the DBDP2 scheme in terms of smaller relative error. Reducing the maturity and increasing the dimension from $d=1$ to $d=10$ and $d=20$ maintains the winning performance. For our unbounded example $(6.2)$, the DADM scheme attains the highest accuracy in tests on both $d=1$ and $d=8$. However, we note that all methods perform poorly when the dimension is increased beyond $d=10$.
\end{subsection}

\

\end{section}

\printbibliography

\end{document}